\documentclass[12pt,reqno]{amsart}
\usepackage{amsfonts}
\usepackage{graphicx}

\newtheorem{thm}{Theorem}[section]
\newtheorem{cor}[thm]{Corollary}
\newtheorem{lem}[thm]{Lemma}

\theoremstyle{definition}

\theoremstyle{remark}
\newtheorem{rem}[thm]{Remark}
\numberwithin{equation}{section}

\begin{document}
\title[combinatorial identities]{On some combinatorial identities and harmonic sums}
\author{Necdet Bat{\i}r}
\address{department of mathematics\\
faculty of sciences and arts\\
nev{\c{s}}eh{\i}r hac{\i} bekta{\c{s}} veli university, nev{\c{s}}eh\i r, turkey}
\email{nbatir@hotmail.com}
\email{nbatir@nevsehir.edu.tr}
\date{December 18, 2015}
\dedicatory{(This paper is a revised version of my paper\\
 to appear in Int. J. Number Theory)}

\begin{abstract}
For any $m,n\in\mathbb{N}$ we first give new proofs for  the following well known  combinatorial identities
\begin{equation*}
S_n(m)=\sum\limits_{k=1}^n\binom{n}{k}\frac{(-1)^{k-1}}{k^m}=\sum\limits_{n\geq r_1\geq r_2\geq...\geq r_m\geq 1}\frac{1}{r_1r_2\cdots r_m}
\end{equation*}
and
$$
\sum\limits_{k=1}^n(-1)^{n-k}\binom{n}{k}k^n = n!,
$$
and then we  produce the  generating function and an integral representation for $S_n(m)$. Using them we evaluate many  interesting finite and infinite harmonic sums in closed form. For example, we show that
$$
\zeta(3)=\frac{1}{9}\sum\limits_{n=1}^\infty\frac{H_n^3+3H_nH_n^{(2)}+2H_n^{(3)}}{2^n},
$$
and
$$
\zeta(5)=\frac{2}{45}\sum\limits_{n=1}^{\infty}\frac{H_n^4+6H_n^2H_n^{(2)}+8H_nH_n^{(3)}+3\left(H_n^{(2)}\right)^2+6H_n^{(4)}}{n2^n},
$$
where $H_n^{(i)}$ are generalized harmonic numbers defined below.
\end{abstract}

\maketitle
\section{Introduction}

We see the following combinatorial sum from time to time in the literature.
\begin{equation}\label{e:1}
S_n(\alpha)=\sum\limits_{k=1}^n\binom{n}{k}\frac{(-1)^{k-1}}{k^\alpha}.
\end{equation}
The identity
\begin{equation}\label{e:2}
S_n(m)=\sum\limits_{n\geq r_1\geq r_2\geq...\geq r_m\geq 1}\frac{1}{r_1r_2\cdots r_m},\quad m\in\mathbb{N}
\end{equation}
 is usually attributed to Dilcher \cite{18}, though Olds \cite{32} gave an equivalent result almost sixty years earlier as a problem solution. Olds proved that for $m\in\mathbb{N}$
 \begin{equation}\label{e:3}
 S_n(m)=\sum\limits_{r_1=1}^n\frac{1}{r_1}\sum\limits_{r_2=1}^{r_1}\frac{1}{r_2}\cdots\sum\limits_{r_{m-1}=1}^{r_{m-2}}\frac{1}{r_{m-1}}\sum\limits_{r_m=1}^{r_{m-1}}\frac{1}{r_m},
 \end{equation}
 which is equivalent to (\ref{e:2}).  In fact the history of results closeley related to Eq. (\ref{e:2}) goes back to Euler's time \cite{19}. The special case of (\ref{e:2})
 $$
 \sum\limits_{k=1}^n(-1)^{k-1}\binom{n}{k}\frac{1}{k}=H_n
 $$
 was given by Euler \cite{19}; see also \cite{24}. In \cite{28} the authors present another generalization of this well known identity. Flajolet and Sedgewick \cite{20},  using residue theorem in complex analysis, showed that $S_n(m)$ can be expressed in terms of the generalized harmonic numbers as
\begin{equation*}
 S_n(m)=\sum\limits_{m_1+2m_2+\cdots=m}\frac{1}{m_1!m_2!\cdots}\left(\frac{H_n^{(1)}}{1}\right)^{m_1}\left(\frac{H_n^{(2)}}{2}\right)^{m_2}\left(\frac{H_n^{(3)}}{3}\right)^{m_3}\cdots,
\end{equation*}
where $H_n^{(r)}$ are the generalized harmonic numbers defined by
$$
H_0^{(r)}=0, \quad H_n^{(1)}=H_n \quad \mbox{and}\quad H_n^{(r)}=\sum\limits_{k=1}^n\frac{1}{k^r}
$$
with $H_n$ the ordinary harmonic numbers; see \cite{21,36}. Connon \cite{15} proved that
$$
S_n(r)=\frac{1}{r!}Y_r\left(0!H_n^{(1)},1!H_n^{(2)},\cdots,(r-1)!H_n^{(r)}\right),
$$
where $Y_r(x_1,x_2,\cdots,x_n)$ are modified Bell polynomials, which can be defined as
$$
Y_0=1\quad \mbox{and} \quad \exp\left(\sum\limits_{k=1}^\infty x_k\frac{t^k}{k!}\right)=\sum\limits_{n=0}^\infty Y_n(x_1,x_2,\cdots,x_n)\frac{t^n}{n!},
$$
see \cite{15,20}. The sum $S_n(r)$ can also be expressed in terms of complete symmetric functions:
$$
S_n(r)=Q_r(p_1,p_2,...,p_r),
$$
where $Q_r$ are the polynomials that express the complete symmetric functions $h_r$ in terms of the power-sum symetric functions $p_i$, i.e.,
$$
h_r=Q_r(p_1,p_2,...p_r).
$$
For definitions and formulas see the first chapter of \cite{31}, particularly p. 28.

The numbers $S_n(m)$ have applications in mathematics. Buchta \cite{8} has shown that $S_n(m-1)$ equals the expected number of maxima of $n$ vectors in $m$-dimensional space, a problem of interest in computational geometry.
In 1775 Euler \cite[p.252]{5} discovered the following elegant series representations:
$$
\zeta(3)=\frac{1}{2}\sum\limits_{n=1}^\infty\frac{H_n}{n^2}\quad \mbox{and}\quad \zeta(4)=\frac{5}{4}\sum\limits_{n=1}^\infty\frac{H_n}{n^3},
$$
where $\zeta$ is the Riemann zeta function. Since then many interesting finite and infinite sums involving generalized harmonic numbers  have been evaluated by many authors by using different techniques. For example Chu \cite {11}:
$$
\zeta(4)=\frac{1}{6}\sum\limits_{n=1}^\infty\frac{H_n^2+H_n^{(2)}}{n^2},
$$
and Adamchik \cite{1}:
\begin{equation*}
\sum\limits_{k=1}^n\frac{H_k^2+H_k^{(2)}}{k}=\frac{H_n^3+3H_nH_n^{(2)}+2H_n^{(3)}}{3}.
\end{equation*}
For many other harmonic sum identities please refer to [9-13, 17, 21, 34-36, 39] and the references given there.

In this paper we present a new approach to evaluate some classes of finite and infinite harmonic sums in closed form. We first provide a new proof of identity (\ref{e:3}) and then we  produce a generating function and an integral representation for $S_n(m)$. Using them we evaluate many  interesting  finite and infinite harmonic sums in closed form, some of which are new and some of which recover known identities. Our results also include some combinatorial identities as special cases of more general identities. In \cite{4} Bang proposed the following problem: Show that
\begin{equation}\label{e:4}
\sum\limits_{k=1}^n(-1)^{k-1}\binom{n}{k}\frac{1}{k}\sum\limits_{j=1}^k\frac{H_j}{j}=H_n^{(3)}.
\end{equation}
In \cite{25} Guo and Qi provided an inductive proof of this identity, which consists of three pages. In \cite{38} Sun and Zhao give the following identity
\begin{equation}\label{e:5}
\sum\limits_{k=1}^{n}(-1)^{k-1}\binom{n}{k}\frac{H_k}{k}=H_n^{(2)}.
\end{equation}
Our results contain identities  (\ref{e:4}) and (\ref{e:5}) as special cases.

The interesting formula
\begin{equation}\label{e:6}
\sum\limits_{k=0}^n(-1)^{n-k}\binom{n}{k}k^m=\left\{
\begin{array}{ll}
0; & \mbox{if} \quad m<n, \\
&  \\
n!; & \mbox{if} \quad m=n.%
\end{array}%
\right.
\end{equation}%
is usually known as Boole's formula in literature because it appears in Boole's book \cite{6}, but actually its history is very old and  goes back to Euler's time. In \cite{22} Gould provides a nice and though discussion of identity (\ref{e:6}), calling it\textit{ Euler's formula}. In 2005 Anglani and Barile \cite{3} give two different proofs of (\ref{e:6}).  In 2008 Phoata \cite{33} offers an extension of (\ref{e:6}) by employing Lagrange's interpolating polynomial theorem, which can be read as
\begin{equation}\label{e:7}
\sum\limits_{k=0}^n(-1)^{n-k}\binom{n}{k}P(a+kb)=a_0.b^n.n!\quad (n\in\mathbb{N}),
\end{equation}
where $a$ and  $b$ are real numbers with $b\neq0$, and $P$ is a polynomial of degree $n$ with leading coefficient $a_0$. In 2009 Katsuura \cite{29} proved that
\begin{equation}\label{e:8}
\sum\limits_{k=0}^n(-1)^k\binom{n}{k}(xk+y)^m=\left\{
\begin{array}{ll}
0; & \mbox{if} \quad 0\leq m<n, \\
&  \\
(-1)^nx^nn!; & \mbox{if} \quad m=n,%
\end{array}%
\right.
\end{equation}%
where $x$ and $y$ are real or complex numbers and $m$ and $n$ are any positive integers. Clearly, Katsuura's result is a special case of Phoata's result. Both identities (\ref{e:7}) and (\ref{e:8}) are not new and they appear in a more general form in \cite{23}. Namely, for any polynomial $f(t)=c_0+c_1t+...+c_mt^m$
of degree m, Gould's entry (Z.8)\cite{23} says that
\begin{equation}\label{e:9}
\sum\limits_{k=0}^n(-1)^k\binom{n}{k}f(k)=\left\{
\begin{array}{ll}
0; & \mbox{if} \quad m<n, \\
&  \\
(-1)^nn!c_n; & \mbox{if} \quad m=n.%
\end{array}%
\right.
\end{equation}%
In fact none of the identities (\ref{e:7})-(\ref{e:9}) are new and it can be easily shown that  they are simple consequences of (\ref{e:6}). In the very recent paper \cite{2} Alzer and Chapman provided a short and new proof and a new  extension of (\ref{e:6}). Our second aim in this work is to provide a new proof of (\ref{e:6}).

Throughout this paper, we shall use the following identities and definitions involving binomial coefficients $\binom{n}{k}$, the gamma function $\Gamma$, beta function $B$, polygamma functions $\psi^{(n)}$ and polylogarithms $Li_n$:
\begin{equation}\label{e:10}
\binom{n+1}{k}=\frac{n+1}{k}\binom{n}{k-1},
\end{equation}
\begin{equation}\label{e:11}
  \binom{n+1}{k}=\binom{n}{k}+\binom{n}{k-1}.
\end{equation}
\begin{equation}\label{e:12}
B(s,t)=\int\limits_0^1u^{s-1}(1-u)^{t-1}du=\frac{\Gamma(s)\Gamma(t)}{\Gamma(s+t)},
\end{equation}
see \cite[Theorem 2,pg.65]{37},
\begin{equation*}
  \psi(x)=\frac{d\thinspace\log\Gamma(x)}{dx}=\frac{\Gamma'(x)}{\Gamma(x)},
\end{equation*}
\begin{equation}\label{e:13}
\psi^{(m)}(n+1)-\psi^{(m)}(1)=(-1)^mm!H_n^{(m+1)},\quad n=0,1,2,...,
\end{equation}
which can be deduced from the well known relation
\begin{equation*}
\psi^{(m)}(x+1)-\psi^{(m)}(x)=\frac{(-1)^mm!}{x^{m+1}},
\end{equation*}
and finally
\begin{equation}\label{e:14}
Li_n(x)=\sum\limits_{k=1}^{\infty}\frac{x^k}{k^n},
\end{equation}
which is valid for $|x|\leq1$ if $n>1$ and $|x|<1$ if $n=1$; see \cite{30} for details.

For the proofs we need the following simple but useful lemma.
\begin{lem} Let $(a_n)$ be any real or complex sequence. Then we have
\begin{equation}\label{e:15}
\sum\limits_{k=1}^n\frac{1}{k}\sum\limits_{j=1}^k(-1)^j\binom{k}{j}a_j=\sum\limits_{k=1}^n(-1)^k\binom{n}{k}\frac{a_k}{k}.
\end{equation}
\end{lem}
\begin{proof}
Clearly (\ref{e:15}) is true for $n=1$. We assume that it is also true for $n-1$ and show that it is true for $n$. By the assumption it suffices to show that

\begin{equation*}
\frac{1}{n}\sum\limits_{j=1}^n(-1)^j\binom{n}{j}a_j=\sum\limits_{k=1}^n(-1)^k\binom{n}{k}\frac{a_k}{k}-\sum\limits_{k=1}^{n-1}(-1)^k\binom{n-1}{k}\frac{a_k}{k}
\end{equation*}
for $n\geq2$, and by the help of (\ref{e:10}) and (\ref{e:11}) this latter equation  is equivalent to
\begin{align*}
\frac{1}{n}\sum\limits_{k=1}^n(-1)^k\binom{n}{k}a_k&=\frac{1}{n}\sum\limits_{k=1}^{n-1}(-1)^k\binom{n}{k}a_k+(-1)^n\frac{a_n}{n}\nonumber\\
&=\sum\limits_{k=1}^{n-1}(-1)^k\binom{n-1}{k-1}\frac{a_k}{k}+(-1)^n\frac{a_n}{n}\nonumber\\
&=\sum\limits_{k=1}^{n-1}(-1)^k\left(\binom{n}{k}-\binom{n-1}{k}\right)\frac{a_k}{k}+(-1)^n\frac{a_n}{n}\nonumber\\
&=\sum\limits_{k=1}^n(-1)^k\binom{n}{k}\frac{a_k}{k}-\sum\limits_{k=1}^{n-1}(-1)^k\binom{n-1}{k}\frac{a_k}{k},\nonumber\\
\end{align*}
which completes the proof.
\end{proof}

\section{Combinatorial identities}
In this section we give new proofs for the combinatorial identities (\ref{e:2}) and (\ref{e:6}).
\begin{thm}
Let $n,m\in\mathbb{N}$. Then we have
\begin{equation}\label{e:16}
S_n(m)=\sum\limits_{r_1=1}^n\frac{1}{r_1}\sum\limits_{r_2=1}^{r_1}\frac{1}{r_2}\cdots\sum\limits_{r_{m-1}=1}^{r_{m-2}}\frac{1}{r_{m-1}}\sum\limits_{r_m=1}^{r_{m-1}}\frac{1}{r_m}.
\end{equation}
\end{thm}
\begin{proof} By mathematical induction on $m$, it suffices to prove
\begin{equation}\label{e:17}
S_n(0)=1 \quad \mbox{and}\quad S_n(m)=\sum\limits_{k=1}^n\frac{S_{k}(m-1)}{k}.
\end{equation}
From (\ref{e:1}) we immediately conclude that $S_n(0)=1$. For $m=1,2,3,\cdots$ we have by Lemma 1.1
\begin{align*}
S_n(m)&=\sum\limits_{k=1}^{n}(-1)^{k-1}\binom{n}{k}\frac{1}{k^m}\\
&=\sum\limits_{k=1}^{n}\frac{1}{k}\sum\limits_{j=1}^{k}(-1)^{j-1}\binom{k}{j}\frac{1}{j^{m-1}}\\
&=\sum\limits_{k=1}^n\frac{S_k(m-1)}{k}.
\end{align*}
\end{proof}
\begin{cor} For all $n,m\in\mathbb{N}$, we have
\begin{equation}\label{e:18}
\sum\limits_{k=1}^n(-1)^{k}\binom{n}{k}\frac{S_k(m)}{k}=H_n^{(m+1)}.
\end{equation}
\end{cor}
\begin{proof}
Applying the inversion formula, see \cite{7,14}
\begin{equation*}
a_n=\sum\limits_{k=1}^n(-1)^k\binom{n}{k}b_k \quad\Leftrightarrow \quad b_n=\sum\limits_{k=1}^n(-1)^k\binom{n}{k}a_k
\end{equation*}
 to
$$
S_n(m)=\sum\limits_{k=1}^n\binom{n}{k}\frac{(-1)^{k-1}}{k^m},
$$
we obtain
\begin{equation*}
\sum\limits_{j=1}^k(-1)^{j-1}\binom{k}{j}S_j(m)=\frac{1}{k^m}.
\end{equation*}
Summing over $k$, after multiplying by $1/k$ both sides of this equation, we find by the help of Lemma 1.1
\begin{align*}
\sum\limits_{k=1}^n\frac{1}{k^{m+1}}=\sum\limits_{k=1}^n\frac{1}{k}\sum\limits_{j=1}^k(-1)^{j-1}\binom{k}{j}S_j(m)=\sum\limits_{k=1}^n(-1)^{k-1}\binom{n}{k}\frac{S_k(m)}{k}.
\end{align*}
\end{proof}
Using the above  inversion formula again we obtain from this identity
$$
\sum\limits_{k=1}^n(-1)^{k-1}\binom{n}{k}H_k^{(m+1)}=\frac{S_n(m)}{n}.
$$
For $m=1$ and $m=2$ we obtain from (\ref{e:18}) identities (\ref{e:4}) and (\ref{e:5}), respectively.

The next theorem provides a new proof of (\ref{e:6}).
\begin{thm} For any $n\in\mathbb{N}$, we have
\begin{equation}\label{e:19}
\sum\limits_{k=0}^n(-1)^{n-k}\binom{n}{k}k^m=\left\{
\begin{array}{ll}
0; & \mbox{if} \quad m<n, \\
&  \\
n!; & \mbox{if} \quad m=n.%
\end{array}%
\right.
\end{equation}%
\end{thm}
\begin{proof}
Let's define, for any $x\in \mathbb{R}$
$$
f_{n,m}(x)=\sum\limits_{k=0}^n(-1)^k\binom{n}{k}(x-k)^m.
$$
for nonnegative integers $m\leq n$. Evidently $f_{0,0}(x)=1$ and $f_{n,0}=0$ for $n\geq 1$ by the binomial theorem. We shall prove that
\begin{equation*}
f_{n,m}(x)=\left\{
\begin{array}{ll}
0; & \mbox{if} \quad m<n, \\
&  \\
n!; & \mbox{if} \quad m=n%
\end{array}%
\right.
\end{equation*}%
by induction on $n$, and equation (\ref{e:19}) follows upon setting  $x=0$. Suppose $n\geq 1$ and the result holds when the first subscript is less than $n$. If $m\geq 1$, then
\begin{align}\label{e:20}
f_{n,m}(x)&=\sum\limits_{k=0}^n(-1)^k\binom{n}{k}x(x-k)^{m-1}-\sum\limits_{k=1}^n(-1)^k\binom{n}{k}k(x-k)^{m-1} \notag\\
&=xf_{n,m-1}(x)-n\sum\limits_{k=1}^n(-1)^k\binom{n-1}{k-1}(x-k)^{m-1}\quad (\mbox{by}\, (1.10)) \notag\\
&=xf_{n,m-1}(x)+n\sum\limits_{k=0}^{n-1}(-1)^k\binom{n-1}{k}(x-k-1)^{m-1}\notag\\
&=xf_{n,m-1}(x)+nf_{n-1,m-1}(x-1).
\end{align}
Now if $m<n$, the induction hypothesis applied to Eq. (\ref{e:20}) implies
$$
f_{n,m}(x)=xf_{n,m-1}(x),
$$
and we can iterate to get
$$
f_{n,m}(x)=x^mf_{n,0}(x)=0.
$$
On the other hand, if $m=n$ then  Eq. (\ref{e:20}) is
$$
f_{n,n}(x)=xf_{n,n-1}(x)+nf_{n-1,n-1}(x-1)= n!
$$
by the preceding case and the induction hypothesis.
\end{proof}
\section{Integral representations and generating functions}
In this section we derive integral representations and generating functions for the numbers  $S_n(m).$
\begin{thm} For all $n,m\in\mathbb{N}$, we have the following integral representation
\begin{equation}\label{e:21}
S_n(m)=\frac{(-1)^{m-1}}{(m-1)!}\int\limits_0^1(1-t^n)\thinspace\frac{\log^{m-1}(1-t)}{1-t}\thinspace dt.
\end{equation}
\end{thm}
\begin{proof}
Using the simple formula
$$
\frac{1}{k^m}=\frac{1}{(m-1)!}\int\limits_0^\infty t^{m-1}e^{-kt}\thinspace dt,
$$
we get
$$
S_n(m)=\sum\limits_{k=1}^n(-1)^{k-1}\binom{n}{k}\frac{1}{k^m}=\frac{-1}{(m-1)!}\sum\limits_{k=1}^n(-1)^k\binom{n}{k}\int\limits_0^\infty t^{m-1}e^{-kt}\thinspace dt.
$$
Inverting the orders of integration and summation, this yields
\begin{align*}
S_n(m)&=\frac{-1}{(m-1)!}\int_{0}^{\infty}t^{m-1}\sum_{k=1}^{n}(-1)^k\binom{n}{k}e^{-kt}\thinspace dt\\
&=\frac{-1}{(m-1)!}\int_{0}^{\infty}[(1-e^{-t})^n-1]t^{m-1}\thinspace dt \notag\\
&=\frac{(-1)^{m-1}}{(m-1)!}\int_{0}^{1}\frac{1-(1-u)^n}{1-(1-u)}\thinspace \log^{m-1}u\thinspace du\quad (\mbox{by setting}\quad u=e^{-t}),
\end{align*}
which completes the proof by the change of variable $1-u=t.$
\end{proof}
Employing the integral representation (\ref{e:21}) and identity (\ref{e:12}) we can evaluate the numbers  $S_n(m)$ .
\begin{cor} For all $n,m\in\mathbb{N}$, we have
\begin{equation}\label{e:22}
S_n(m)=\frac{(-1)^{m}n!}{m!}G^{(m)}(1),
\end{equation}
where
$$
G(s)=\frac{\Gamma(s)}{\Gamma(n+s)}.
$$
\end{cor}
\begin{proof}
By integration by parts, (\ref{e:21}) becames
\begin{align*}
S_n(m)&=\frac{(-1)^{m}n}{m!}\int\limits_0^1t^{n-1}\log^m(1-t)dt\\
&=\frac{(-1)^{m}n}{m!}\int\limits_0^1t^{n-1}\frac{\partial^m}{\partial s^m}(1-t)^{s-1}\bigg|_{s=1}dt\\
&=\frac{(-1)^{m}n}{m!}\frac{\partial^m}{\partial s^m}\int\limits_0^1t^{n-1}(1-t)^{s-1}dt\bigg|_{s=1}\\
&=\frac{(-1)^{m}n}{m!}\frac{\partial^m}{\partial s^m}B(n,s)\bigg|_{s=1}\\
&=\frac{(-1)^{m}n}{m!}\frac{\partial^m}{\partial s^m}\frac{\Gamma(n)\Gamma(s)}{\Gamma(n+s)}\bigg|_{s=1},
\end{align*}
which completes the proof.
\end{proof}
The first values of $S_n(m)$ can be computed  by using formula (\ref{e:22}) and Eq. (\ref{e:13}) as follows.
\begin{equation}\label{e:23}
S_n(1)=H_n,
\end{equation}
\begin{equation}\label{e:24}
S_n(2)=\frac{H_n^2}{2}+\frac{H_n^{(2)}}{2},
\end{equation}
\begin{equation}\label{e:25}
S_n(3)=\frac{H_n^3}{6}+\frac{H_nH_n^{(2)}}{2}+\frac{H_n^{(3)}}{3},
\end{equation}
\begin{equation}\label{e:26}
S_n(4)=\frac{1}{24}\bigg\{H_n^4+6H_n^2H_n^{(2)}+8H_nH_n^{(3)}+3\left(H_n^{(2)}\right)^2+6H_n^{(4)}\bigg\}
\end{equation}
and
\begin{align}\label{e:27}
S_n(5)&=\frac{1}{120}\bigg\{H_n^5+10H_n^3H_n^{(2)}+20H_n^2H_n^{(3)}+15H_n\left(H_n^{(2)}\right)^2\notag \\
&+30H_nH_n^{(4)}+20H_n^{(2)}H_n^{(3)}+24H_n^{(5)}\bigg\}.
\end{align}
The first three of these identities are already known from the paper by Flajolet and Sedgewick \cite[p.108]{20}.

Now we derive the generating function of $S_n(m)$.
\begin{thm}For $m\in\mathbb{N}$ and $-1<x\leq1/2$ it holds that
\begin{equation}\label{e:28}
\sum\limits_{n=1}^\infty S_n(m)x^n=-\frac{1}{1-x}Li_m\left(-\frac{x}{1-x}\right),
\end{equation}
where $Li_m$ is polylogarithm function defined by (\ref{e:14}).
\end{thm}
\begin{proof} Let $(a_n)$ be any sequence,

\begin{equation}\label{e:29}
b_n=\sum\limits_{k=0}^n(-1)^k\binom{n}{k}a_k,
\end{equation}
and $F$ be the generating function of $(a_n)$, namely, $F(x)=\sum_{n=0}^{\infty}a_nx^n$, then from \cite[p.103]{20} we know that generating function $G$ of $(b_n)$ is
$$
G(x)=\frac{1}{1-x}F\left(-\frac{x}{1-x}\right).
$$
Taking $a_k=1/k^m$ in (\ref{e:29}), we find that generating function of $S_n(m)$ is
$$
-\frac{1}{1-x}Li_m\left(-\frac{x}{1-x}\right).
$$
\end{proof}
Integrating both sides of (\ref{e:28}), after multiplying by $1/x$, we get
\begin{equation*}
\sum\limits_{n=1}^\infty S_n(m)\frac{x^n}{n}=-\int\limits_0^x\frac{Li_m(-t/(1-t))}{t(1-t)}\thinspace dt.
\end{equation*}
By the change of variable $t=-\frac{u}{1-u}$, we get after some simple computations for $-1\leq x\leq \frac{1}{2}$
\begin{equation}\label{e:30}
\sum\limits_{n=1}^\infty S_n(m)\frac{x^n}{n}=-Li_{m+1}\left(-\frac{x}{1-x}\right).
\end{equation}
\section{Applications}
Taking some particular values for $m$ and $x$ in Eqs. (\ref{e:28}) and (\ref{e:30}), we can evaluate many interesting series representations for the Riemann zeta function $\zeta$ and finite harmonic sums.
For $x=1/2$ , we get from (\ref{e:28})
\begin{equation*}
\sum\limits_{n=1}^\infty\frac{S_n(m)}{2^n}=-2Li_m(-1)\quad m=2,3,4\cdots.
\end{equation*}
Since
\begin{equation*}
Li_m(-1)=\frac{1-2^{m-1}}{2^{m-1}}\thinspace\zeta(m),
\end{equation*}
we find
\begin{equation}\label{e:31}
\zeta(m)=\frac{2^{m-2}}{2^{m-1}-1}\sum\limits_{n=1}^\infty\frac{S_n(m)}{2^n}\quad m=2,3,4\cdots
\end{equation}
with $\zeta$ be the Riemann zeta function.

For  $m=2$, we get from (\ref{e:31}) by taking into account (\ref{e:24})
\begin{equation*}
\sum\limits_{n=1}^\infty\frac{H_n^2+H_n^{(2)}}{2^n}=2\zeta(2).
\end{equation*}
For $m=3$, we get from (\ref{e:31}) by the help of (\ref{e:25}) the following new series representation for the Apery constant $\zeta(3)$
\begin{equation*}
\zeta(3)=\frac{1}{9}\sum\limits_{n=1}^\infty\frac{H_n^3+3H_nH_n^{(2)}+2H_n^{(3)}}{2^n}.
\end{equation*}
For $x=1/2$, we get from (\ref{e:30})
\begin{equation}\label{e:32}
\zeta(m+1)=\frac{2^m}{2^m-1}\sum\limits_{n=1}^\infty\frac{S_n(m)}{n2^n} \quad m=1,2,3,\cdots.
\end{equation}
For $m=1$  we find from  (\ref{e:32})
\begin{equation}\label{e:33}
\sum\limits_{n=1}^\infty\frac{H_n}{n2^n}=\frac{1}{2}\zeta(2).
\end{equation}
For  $m=2$  we find from  (\ref{e:32}) the following  series representation for the Apery constant $\zeta(3)$
\begin{equation}\label{e:34}
\zeta(3)=\frac{2}{3}\sum\limits_{n=1}^\infty\frac{H_n^2+H_n^{(2)}}{n2^n}.
\end{equation}
For $x=1/2$ and $m=3$  we find from (\ref{e:32})
\begin{equation}\label{e:35}
\zeta(4)=\frac{4}{21}\sum\limits_{n=1}^\infty\frac{H_n^3+3H_nH_n^{(2)}+2H_n^{(3)}}{n2^n}.
\end{equation}
The last three identities are known and appear in \cite{10}; see Eqs. (4.25), (4.24), (4.26) and (4.27).
Taking $x=-1$ in (\ref{e:30})
we get
\begin{equation}\label{e:36}
\sum\limits_{n=1}^\infty(-1)^n\frac{S_n(m)}{n}=-Li_{m+1}(1/2)
\end{equation}
Lettting $m=1$ in (\ref{e:36})  
and taking into account $Li_2(1/2)=\zeta(2)/2-\frac{1}{2}\log^22$, we get
\begin{equation}\label{e:37}
\sum\limits_{n=1}^\infty(-1)^{n-1}\frac{H_n}{n}=\frac{1}{2}(\zeta(2)-\log^22).
\end{equation}
For $m=2$  we get from  (\ref{e:36})
\begin{equation*}
\sum\limits_{n=1}^\infty\frac{(-1)^n\left(H_n^2+H_n^{(2)}\right)}{n}=-2Li_3(1/2).
\end{equation*}
Using $Li_3(1/2)=\frac{7}{8}\zeta(3)+\frac{1}{2}\zeta(2)\log(\frac{1}{2})-\frac{1}{6}\log^3(\frac{1}{2})$, see \cite[p.2]{18}, this is
\begin{equation}\label{e:38}
\sum\limits_{n=1}^\infty\frac{(-1)^{n-1}\left(H_n^2+H_n^{(2)}\right)}{n}=\frac{7}{4}\zeta(3)-\zeta(2)\log 2+\frac{1}{3}\log^3 2.
\end{equation}
Let $x=1/2$ and $m=1$ in (\ref{e:28}), to get the following  well known identity
\begin{equation}\label{e:39}
\sum\limits_{n=1}^\infty\frac{H_n}{2^n}=-2Li_1(-1)=2\log2.
\end{equation}
Eq. (\ref{e:39}) is well known; see \cite{28} but Eq. (\ref{e:38}) seems to be new.
Taking  $m=4$ in (\ref{e:32}), we obtain by using (\ref{e:26})
\begin{equation*}
\zeta(5)=\frac{2}{45}\sum\limits_{n=1}^{\infty}\frac{H_n^4+6H_n^2H_n^{(2)}+8H_nH_n^{(3)}+3\left(H_n^{(2)}\right)^2+6H_n^{(4)}}{n2^n}.
\end{equation*}
Let $\rho=\frac{\sqrt{5}-1}{2}$ be the inverse of the golden ratio. Then for $x=-\rho$ and $m=2$, we get from (\ref{e:30}) by using (\ref{e:24})
\begin{equation*}
\sum_{n=1}^{\infty}\frac{H_n^2+H_n^{(2)}}{n}(-\rho)^n=-2Li_3(\rho^2).
\end{equation*}
Using $Li_3(\rho^2)=\frac{4}{5}\zeta(3)+\frac{4}{5}\zeta(2)\log(\rho)-\frac{2}{3}\log^3(\rho)$; see \cite[p.2]{30}, we get
$$
\sum_{n=1}^{\infty}\frac{H_n^2+H_n^{(2)}}{n}(-\rho)^n=-\frac{8}{5}\zeta(3)-\frac{8}{5}\zeta(2)\log(\rho)+\frac{4}{3}\log^3(\rho).
$$
Taking $x=-\rho$ and $m=3$ in (\ref{e:28}), we find
\begin{equation*}
\sum_{n=1}^{\infty}\left(H_n^3+3H_nH_n^{(2)}+2H_n^{(3)}\right)(-\rho)^{n-1}=\frac{24}{5}\zeta(3)+\frac{24}{5}\zeta(2)\log \rho-4\log^3 \rho.
\end{equation*}
From Theorem 2.1 we have that for all $m,n\in\mathbb{N}$
\begin{equation}\label{e:40}
S_n(0)=1 \quad \mbox{and} \quad S_n(m)=\sum\limits_{k=1}^n\frac{S_k(m-1)}{k}.
\end{equation}
Applying this identity and Eqs. (\ref{e:23})-(\ref{e:27}) we get for $m=2,3,4,5$ the following finite harmonic sum identities, respectively:

\begin{equation}\label{e:41}
\sum\limits_{k=1}^n\frac{H_k}{k}=\frac{H_n^2+H_n^{(2)}}{2},
\end{equation}

\begin{equation}\label{e:42}
\sum\limits_{k=1}^n\frac{H_k^2+H_k^{(2)}}{k}=\frac{H_n^3+3H_nH_n^{(2)}+2H_n^{(3)}}{3}
\end{equation}

\begin{align}\label{e:43}
\sum\limits_{k=1}^n&\frac{H_k^3+3H_kH_k^{(2)}+2H_k^{(3)}}{k}\notag\\
&=\frac{H_n^4+6H_n^2H_n^{(2)}+8H_nH_n^{(3)}+3\left(H_n^{(2)}\right)^2+6H_n^{(4)}}{4},
\end{align}
and
\begin{align}\label{e:44}
\sum\limits_{k=1}^n&\frac{H_k^4+6H_k^2H_k^{(2)}+8H_kH_k^{(3)}+3\left(H_k^{(2)}\right)^2+6H_k^{(4)}}{k}\notag\\
&=\frac{1}{5}\bigg\{H_n^5+10H_n^3H_n^{(2)}+20H_n^2H_n^{(3)}+15H_n\left(H_n^{(2)}\right)^2.\notag \\
&+30H_nH_n^{(4)}+20H_n^{(2)}H_n^{(3)}+24H_n^{(5)}\bigg\}.
\end{align}
Identities (\ref{e:41}) and (\ref{e:42}) are known and can be found in \cite{1}. Eqs. (\ref{e:31}) and (\ref{e:32}) can be compared to some of the results given in \cite{26}:
$$
\sum\limits_{k=1}^\infty\frac{S_n(m)}{n^2}=(m+1)\zeta(m+2), m\geq 0,
$$
$$
\sum\limits_{k=1}^\infty\frac{S_n(m)}{n(n+1)}=(m+1)\zeta(m+1), m\geq 1,
$$
and
$$
\sum\limits_{k=1}^\infty\frac{S_n(m)}{(n+1)(n+2)}=m\zeta(m), m\geq 2.
$$
The first of these identities also follows from \cite[Cor. 2]{16}.
Our final example can not be deduced from the results above but we think it is curious enough to warrant adding it here.
$$
\sum\limits_{k=1}^{n}\frac{H_kH_{k-1}}{k}=\frac{H_n^3+H_n^{(3)}}{3}.
$$
This can be derived by summing over $k=1,2,3,...,n$ the identity
$$
H_k^3-H_{k-1}^3=\frac{1}{k^3}+\frac{3H_kH_{k-1}}{k}.
$$

\section{Remarks}
\begin{rem}
Numbers like $S_n(m)$ arise from number theory. M. I. Israilov \cite{27} considered the coefficients $\gamma_n$ in the Laurent expansion of the Riemann zeta function $\zeta(s)$ about its pole $s=1$ and found a new expression for $\gamma_n$ including numbers
$$
b_{j,\mu}=\sum\limits_{r_1=j}^\mu\frac{1}{r_1}\sum\limits_{r_2=j-2}^{r_1-1}\frac{1}{r_2}\cdots\sum\limits_{r_j=1}^{r_{j-1}-1}\frac{1}{r_j}.
$$
\end{rem}
\begin{rem} In fact Eq. (3.8) is valid for all integers $m$. If we replace $-m$ by $m$ in Eq. (3.8) we get for $-1<x<1/2$
$$
\sum\limits_{n=1}^\infty S_n(-m)x^n=-\frac{1}{1-x}\sum\limits_{k=1}^\infty k^m\left(\frac{-x}{1-x}\right)^m.
$$
But by Theorem 2.3 we have $S_n(-m)=0$ for $n>m$, and in this case the left side becomes a finite sum. We therefore have
$$
\frac{1}{x-1}\sum\limits_{k=1}^\infty k^m\left(\frac{-x}{1-x}\right)^k=\sum\limits_{n=1}^m S_n(-m)x^n \quad -1<x<1/2.
$$
\end{rem}
\begin{rem}
Taking into account Theorem 2.3 we should notice that the relation is also valid even if we  replace $m$ by any real number.
\end{rem}
\begin{rem}
Identity (\ref{e:28}) is proved by Connon in \cite{15} but his proof is very  long (2 pages).
\end{rem}
\begin{rem}The numbers $S_n(m)=-A_n(m)$ are closely related with the Stirling numbers of the second kind S(n,m) : $S(n,m)=\frac{(-1)}{m!}S_m(-n).$
\end{rem}
\textbf{Acknowledgement} I would like to thank the referee  for his/her thorough review and highly appreciate the comments and
suggestions, which significantly contributed to improving the quality of the publication. Dedicated to great Turkish mathematician Professor \textit{Masatoshi G{\"{u}}nd{\"{u}}z Ikeda} on the occasion
 of his $90$th birthday

\end{document}